\documentclass[12pt, reqno]{amsart}
\usepackage{amsmath, amsthm, amscd, amsfonts, amssymb, graphicx, xcolor}
\usepackage[bookmarksnumbered, colorlinks, plainpages]{hyperref}
\usepackage{mathrsfs}
\usepackage{tikz}
\usetikzlibrary{arrows.meta, calc, decorations.pathreplacing, positioning}

\newtheorem{theorem}{Theorem}[section]
\newtheorem{lemma}[theorem]{Lemma}

\newtheorem{proposition}[theorem]{Proposition}
\newtheorem{corollary}[theorem]{Corollary}
\theoremstyle{definition}
\newtheorem{definition}[theorem]{Definition}

\newtheorem{question}[theorem]{Question}
\theoremstyle{remark}
\newtheorem{remark}[theorem]{Remark}
\numberwithin{equation}{section}

\begin{document}
\setcounter{page}{1}

\newcommand{\ywy}[1]{{\color{red}{#1}}}

\centerline{}

\centerline{}

\title[Geometric Infiniteness in Gromov Hyperbolic Spaces]{Two Characterizations of Geometrically Infinite Actions on Gromov Hyperbolic Spaces}

\author{Chaodong Yang}	
 \address{Beijing International Center for Mathematical Research\\
Peking University\\
 Beijing 100871, China
P.R.}
 \email{be2sio4@stu.pku.edu.cn}
 
\author{Wenyuan Yang}

%Address of record for the research reported here
\address{Beijing International Center for Mathematical Research\\
Peking University\\
 Beijing 100871, China
P.R.}
\email{wyang@math.pku.edu.cn}

\date{February 25th 2026}


\begin{abstract}
    We provide two new characterizations of geometrically infinite actions on Gromov hyperbolic spaces: one in terms of the existence of escaping geodesics, and the other via the presence of uncountably many non-conical limit points. These results extend corresponding theorems of Bonahon, Bishop, and Kapovich--Liu from the settings of Kleinian groups and pinched negatively curved manifolds to discrete groups acting properly on proper Gromov hyperbolic spaces.
\end{abstract}
%\dedicatory{This paper is dedicated to Professor ABCD}

\thanks{(W.Y.) Partially supported by National Key R \& D Program of China (2025YFA1017500) and  National Natural Science Foundation of China (No. 12131009 and No.12326601)}

\maketitle
\section{Introduction}

Let $X$ be a proper Gromov hyperbolic space, on which a countable group $G$ acts properly by isometries. The action naturally extends to a convergence group action on the Gromov boundary $\partial X$, which decomposes $\partial X$ into  the \textit{limit set} $\Lambda G$ and the  domain of discontinuity. Let $\Pi \subset \Lambda G$ be the set of parabolic points.  The action of $G$ on $X$ is called \textit{geometrically finite} if there exists a sufficiently separated $G$-invariant family of horoballs $\{B_p:p\in \Pi\}$  so that the action on the complement $\mathrm{CH}(\Lambda G)\setminus \bigcup_{p\in \Pi} B_p$ is cocompact, where $\mathrm{CH}(\Lambda G)$ denotes the weak convex hull of $\Lambda G$ in $X$.  In this case, $\Lambda G\setminus \Pi$ consists precisely of \textit{conical points}. We refer to Section \ref{preliminary} for precise definitions.

Historically, the notion of geometric finiteness originated in the theory of Kleinian groups (see \cite{BeaMas,Marden,thurstonnotes}) and has since been extended to the much broader class of groups considered here in a series of works \cite{Bow3,Bow4,bowditch2012relatively}. Since the formal introduction in \cite{Gromov}, this class of groups has been intensively studied in geometric group theory under the name of \textit{relatively hyperbolic groups}, encompassing many examples of geometric and algebraic origin. Several equivalent characterizations of relative hyperbolicity are then proposed in the literature, generalizing and going well beyond those known for Kleinian groups; see, for example, \cite{Farb, bowditch2012relatively, Osin, DruSapir, Ge1}, as well as \cite{Hru} for an overview and references therein. 

The aim of this paper is to supplement these descriptions with two new characterizations in the present general setting.  %We formulate it in terms of geometrical infiniteness.

Let $\{g_n\}_{n\ge 1}$ be a sequence of hyperbolic elements in $G$. We say that $\{g_n\}_{n\ge 1}$ is \textit{escaping} if for any point $o$ and any $R>0$, the $R$-neighborhood of $Go$ intersects the axes of only finitely many $g_n$. 

\begin{theorem}\label{geometricallyFiniteNearOrbit}
    Suppose that a group $G$ admits a non-elementary proper action on a proper Gromov hyperbolic space $X$. Then the action is geometrically infinite if and only if there exists an escaping sequence of hyperbolic elements in $G$.
\end{theorem}

The corresponding result for Kleinian groups was proved by Bonahon \cite{bonahon-bouts} and was recently generalized to pinched negatively curved manifolds by Kapovich and Liu \cite{KL19}. In contrast to these works, our approach relies only on elementary arguments in hyperbolic geometry and does not invoke the Margulis lemma.
In detail, we define a region that acts as a horoball (see Section~\ref{horoballConstruction}) and use Lemma~\ref{producingHyperbolicElements} to directly obtain a hyperbolic element of word length $2$, thus circumventing the essential obstacle in previous works.

Following the strategy of Bishop \cite{Bishop96} for Kleinian groups, we derive the second characterization of geometric infiniteness from Theorem \ref{geometricallyFiniteNearOrbit}. This refines the result of Beardon and Maskit \cite{BeaMas} in terms of non-conical limit sets. 

\begin{theorem}\label{geometricallyFiniteNonconicalCountable}
    Suppose that a group $G$ admits a non-elementary proper action on a proper Gromov hyperbolic space $X$. Then the action is geometrically infinite if and only if the set of non-conical points is uncountable.
\end{theorem}

We remark that the Kleinain group case was done by Bishop \cite{Bishop96} and the pinched negatively curved manifold case was established by Kapovich and Liu \cite{KL20} after their work \cite{KL19} on Theorem \ref{geometricallyFiniteNearOrbit}. To put this into perspective, for finitely generated geometrically infinite Kleinian groups, the non‑conical limit set has maximal Hausdorff dimension; its precise value was computed by Bishop–Jones \cite{BJ97bddgeo}, subsequent contributions by \cite{KL20}, and was recently completed in the full general case by \cite{mj2025hausdorff}. The question of whether Theorem~\ref{geometricallyFiniteNonconicalCountable} holds at this level of generality arose naturally in the course of that study.

\begin{remark}
    As a direct corollary of Theorem~\ref{geometricallyFiniteNonconicalCountable}, such an action is geometrically finite if and only if every point of $\Lambda G$ is either conical or parabolic. That is, we can remove the \textit{bounded} assumption on parabolic points. See Definition \ref{defn:limitpoints}.  
\end{remark}

Finally, let us conclude the Introduction with an open question.  
By a theorem of Yaman \cite{Yaman} (and of Bowditch \cite{Bow98} in the cocompact case), geometrically finite actions on Gromov hyperbolic spaces can be recovered from the dynamics on the boundary. More precisely, if a convergence group action is geometrically finite (i.e., its limit set consists only of conical points and bounded parabolic points) then the group admits a geometrically finite action on some proper Gromov hyperbolic space. Moreover, the induced action on the boundary coincides with the original convergence group action. In the same sprint of those works, one can ask  whether the analogue of Theorem \ref{geometricallyFiniteNonconicalCountable} holds in a purely dynamical term of convergence groups. 

\begin{question}\label{geometricallyFiniteConvergenceGroup}
    Suppose that a discrete group $G$ admits a convergence group action on a compact metrizable space with  countably many non-conical limit points. Does $G$ admit a geometrically finite action on some Gromov hyperbolic space?
\end{question}
 
\begin{remark}

\begin{enumerate}
\item
This is related to a question of Kapovich \cite[Problem 7]{Kap05} asking whether every convergence group action can be realized as the boundary action of a proper action on some Gromov hyperbolic space. If so, then a positive answer would follow from Theorem~\ref{geometricallyFiniteNonconicalCountable} together with Yaman’s theorem.
It is also deserved mentioning that in \cite{sun2019dynamicalcharacterizationacylindricallyhyperbolic}, Sun proved that the action constructed in \cite{Bow98} on a hyperbolic space contains WPD (i.e., weak proper discontinuity) elements.  

\item The assumption that the action is a convergence group action is necessary. It is easy to give examples of \emph{non-relatively hyperbolic} groups $G$ which admit a minimal action on a compact metrizable space $M$ such that all but {countably many points} $\xi\in M$ are \emph{conical points} in the usual sense.
\begin{itemize}
    \item  There exist a sequence $\{g_n\}$ in $G$ and distinct points $a, b \in M$ such that $g_n x \to b$ locally uniformly for all $x \in M \setminus \{\xi\}$, while $g_n \xi \to a$.  
\end{itemize}
Indeed, any non-elementary action on a locally infinite tree $T$ gives rise to the desired action on the compact space $M=\partial T\cup T^0$, the union of the ends of the tree and the vertex set. We equip $M$ with Bowditch topology  (since $T$ is a fine graph in the sense of \cite{bowditch2012relatively}). It is easy to check that the ends of the tree are conical points in the above sense.
%(Equivalently, $M$ is the horofunction compactification of $T$ in the pointwise convergence topology.)      
\end{enumerate}
    
\end{remark}
\subsection*{Structure of the paper}
Section 2 reviews the necessary preliminaries on convergence group actions and $\delta$-hyperbolic spaces. A key technical lemma (Lemma~\ref{producingHyperbolicElements}) for producing hyperbolic elements under a dynamical condition is presented.
In Section 3, assuming the condition in Theorem~\ref{geometricallyFiniteNearOrbit}, we construct geometric objects analogous to those in a cusp-uniform action. We prove that every point in the limit set $\Lambda G$ is either conical or bounded parabolic, thereby establishing the theorem.
In Section 4, we largely follow the strategy of \cite{Bishop96} to construct uncountably many non-conical limit points under the hypothesis of geometric infiniteness. The proof uses the tool of $L$-local quasi-geodesics to build the required paths, thus proving Theorem~\ref{geometricallyFiniteNonconicalCountable}.

\subsection*{Acknowledgments}
The second-named author is grateful to Mahan MJ  for insightful discussions about  Theorem \ref{geometricallyFiniteNonconicalCountable}.

\section{Preliminaries}\label{preliminary}
In this section we introduce the preliminaries we will use on convergence group actions (cf. \cite{tukia1994convergence, bowditch1999convergence}) and hyperbolic spaces (cf. \cite{bridson2013metric}). In particular, we follow the account of  \cite{bowditch2012relatively} on geometrically finite actions. 

\subsection{Convergence actions}
\begin{definition}
    Let $M$ be a compact metrizable space and $G$ a discrete group acting on $M$ by homeomorphisms. The action is called a \textit{convergence group action} if, for every sequence $\{g_n\}$ of distinct elements in $G$, there exist a subsequence $\{g_{n_i}\}$ and points $a, b \in M$ such that $g_{n_i} x \to b$ locally uniformly for all $x \in M \setminus \{a\}$.
\end{definition}

\begin{definition}
    Given such an action, the \textit{limit set} $\Lambda H$ of a subgroup $H \subset G$ is defined as the set of all accumulation points of the $H$-orbits in $M$. 
    
    The action is said to be \textit{non-elementary} if $\#\Lambda G > 2$.
\end{definition}

If the action is non-elementary, the limit set $\Lambda G$ is the unique minimal $G$-invariant closed subset; moreover, $\Lambda G$ is a perfect set (i.e., it has no isolated points) and therefore contains uncountably many points.  

%\begin{definition}
    An element of $G$ is called \textit{elliptic} if it has finite order. By the very definition of convergence group action, every element of infinite order has at least one but at most two fixed points in $M$. Among these, an element is termed \textit{hyperbolic} if it has exactly two fixed points, and \textit{parabolic} if it has exactly one.
%\end{definition}

The following lemma from \cite{tukia1994convergence} is used to produce hyperbolic elements in convergence group actions. We modify some conditions and provide a proof here.
\begin{lemma}\label{producingHyperbolicElements}
    Let the action of $G$ on $M$ be a non-elementary convergence group action.
    If there exist an element $f \in G$ and an open set $U \subset M$ such that $f\overline{U} \subsetneqq U$, then $f$ is hyperbolic.
\end{lemma}

\begin{proof}
    We first show that $f$ has infinite order.
    Assume to the contrary that $f^n = \operatorname{id}$ for some $n \in \mathbb{N}$.
    Then $U \subseteq \overline{U} = f^n\overline{U} \subsetneqq U$, a contradiction.

    Define $K = \bigcap\limits_{n=1}^{\infty} f^{n}\overline{U}$.
    Since $M$ is compact, $K$ is nonempty.
    From $f\overline{U} \subset U \subset \overline{U}$, we obtain $fK=K$.

    Next, define $L = \bigcap\limits_{n=1}^{\infty} f^{-n}(M \setminus U)$.
    Observe that $f^{-1}(M \setminus U) \subset M \setminus \overline{U} \subset M \setminus U$.
    Similarly, we have $L\ne \varnothing$ and $f^{-1}L=L$, hence $fL=L$.

    Since $K \cap L \subset f\overline{U} \cap (M \setminus U) = \varnothing$ and $f$ has infinite order, the sets $K$ and $L$ must be distinct fixed points of $f$.
    Therefore, $f$ is hyperbolic.
\end{proof}

\begin{definition}\label{defn:limitpoints}
    A point $\xi \in M$ is called \textit{conical} if there exist a sequence $\{g_n\}$ in $G$ and distinct points $a, b \in M$ such that $g_n x \to b$ locally uniformly for all $x \in M \setminus \{\xi\}$, while $g_n \xi \to a$.
    
    A point $\xi \in M$ is \textit{parabolic} if its stabilizer $G_\xi$ is infinite and its limit set satisfies $\Lambda G_\xi = \{\xi\}$.
    
    A parabolic point $\xi$ is \textit{bounded parabolic} if the action of $G_\xi$ on $M \setminus \{\xi\}$ is cocompact.

\end{definition}
    
    Note that both conical and parabolic points lie in the limit set $\Lambda G$.
    We denote by $\Lambda^{c} G$ and $\Lambda^{nc} G$ the sets of conical and non-conical points in $\Lambda G$, respectively. In particular, the fixed points of any hyperbolic element are conical. The following definition has its origin in the work of Beardon-Maskit \cite{BeaMas} for Kleinian groups.

\begin{definition}\label{defn:geomfiniteconvergence}
    A convergence group action of $G$ on $M$ is \textit{geometrically finite} if every point of $\Lambda G$ is either conical or bounded parabolic.
\end{definition}

Note that the induced action of $G$ on $\Lambda G$ is a convergence group action. In the literature, geometrically finite actions are usually required to be minimal, i.e., $\Lambda G = M$. We deliberately adopt this definition to remain consistent with the notion of a geometrically finite action on a hyperbolic space; see Definition \ref{defn:geomfiniteonhypspace} and Lemma~\ref{geomFinite}.

\subsection{Gromov hyperbolic geometry}
\begin{definition}
    Let $\delta \ge 0$. A geodesic space $X$ is called \textit{$\delta$-hyperbolic} if every geodesic triangle in $X$ is \textit{$\delta$-thin}; that is, each side is contained in the $\delta$-neighborhood of the union of the other two sides.
\end{definition}
In the sequel, we always assume that $X$ is a proper geodesic $\delta$-hyperbolic space.

\begin{definition}
    The set of asymptotic classes of geodesic rays in $X$ is called the \textit{Gromov boundary}, denoted by $\partial X$. We set $\overline{X} := X \cup \partial X$, which forms a compactification of $X$ under the cone topology.
\end{definition}

Note that any two points $x, y \in \overline{X}$ can be joined by a geodesic. We denote by $[x, y]$ a (often arbitrary) choice of geodesic between $x$ and $y$. Given a geodesic $\gamma=[x, y]$ and $x^\prime, y^\prime\in \gamma$, we denote the subpath of $\gamma$ between $x^\prime, y^\prime$ by $[x^\prime, y^\prime]_{\gamma}$. 

For three distinct points $x, y, z \in \overline{X}$, an \textit{ideal triangle} $\triangle(xyz)$ means the union $[x, y] \cup [y, z] \cup [z, x]$ (with endpoints possibly in boundary). In a $\delta$-hyperbolic space, every ideal triangle is $2\delta$-thin. Furthermore, in a $\delta$-hyperbolic space $X$, any quadrilateral is $2\delta$-thin, and any ideal quadrilateral is $4\delta$-thin. These facts are left as exercises for the interested reader.

\begin{definition}
    For three points $x, y, z\in X$, the \textit{Gromov product} $\langle x, y\rangle_z$ is defined as
$$
\langle x, y\rangle_z=\frac{1}{2}\left(d(z, x)+d(z, y)-d(x, y)\right).
$$
\end{definition}
A basic estimate is that for any $x, y, z\in X$,
$$
d([x, y], z)-4\delta \le \langle x, y\rangle_z\le d([x, y], z).
$$

Gromov's four-point condition provides an equivalent definition of $\delta$-hyperbolicity: for any $x, y, z, w\in X$,
$$
\langle x, y\rangle_w\ge \min\left\{ \langle x, z\rangle_w, \langle y, z\rangle_w \right\}-\delta.
$$
From this, one deduces the following useful lemma.
\begin{lemma}\label{fourPoints}
    Suppose $X$ is a $\delta$-hyperbolic space for some $\delta\ge 0$. For any four points $x, y, z, w\in X$,
    $$
    d(w, [x, y])\ge \min \left\{d(w, [x, z]), d(w, [y, z])\right\}-5\delta.
    $$
\end{lemma}

\begin{definition}
    Fix a point $z\in X$. There exists a constant $K\ge 1$ such that for sufficiently small $\varepsilon>0$, we can equip $\overline{X}$ with a metric $\rho_{z, \varepsilon}$ satisfying
    \begin{align}\label{eqn: visual metric}
    \frac{1}{K}\, \mathrm{e}^{-\varepsilon d(z, [x, y])}\le \rho_{z, \varepsilon}(x, y)\le \mathrm{e}^{-\varepsilon d(z, [x, y])},
    \end{align}
    for all $x\ne y \in \overline{X}$. This family of  metrics  are called \textit{visual metrics} based at $z$ with parameter $\varepsilon$. 
\end{definition}
The topology induced on $\partial X$ by any visual metric coincides with the cone topology (see \cite{bridson2013metric}).

Given a path $p$ in $X$, we denote the endpoints of $p$ by $p_-, p_+$. 
\begin{definition}
    Given $\lambda\ge 1, C\ge 0$, a path $p$ in $X$ is called a $(\lambda, C)$\textit{-quasi-geodesic} if $\operatorname{Len}(q)\le \lambda d(q_-, q_+)+C$ for every connected subpath $q$ of $p$. 
    
    Let $L>0, \lambda\ge 1$ and $C\ge 0$. 
    A path $p$ is called an $L$\textit{-local} $(\lambda, C)$\textit{-quasi-geodesic} if every subpath of $p$ of length $L$ is a $(\lambda, C)$-quasi-geodesic.
\end{definition}

We will use the following Morse Lemma for $\delta$-hyperbolic spaces.
\begin{lemma}\cite[Page 401]{bridson2013metric}\label{MorseLemma}
    Given $\lambda\ge 1$ and $C\ge 0$, there exists a constant $B=B(\lambda, C)$ with the following property.

    If $\gamma$ is a $(\lambda, C)$-quasi-geodesic in $X$, then $d_{\mathrm{Haus}}(\gamma, [\gamma_-, \gamma_+])\le B$.
\end{lemma}

\begin{lemma}[\cite{bridson2013metric}, Page 405]\label{localQuasiGeodesic}
    For any $\lambda\ge 1, C\ge 0$, there exist $L_0>0$, $\lambda^\prime\ge 1$ and $C^\prime\ge 0$ with the following property.
    Given $L\ge L_0$, let $\gamma$ be an $L$-local $(\lambda, C)$-quasi-geodesic. 
    Then $\gamma$ is a $(\lambda^\prime, C^\prime)$-quasi-geodesic.
\end{lemma}

The isometries of a Gromov hyperbolic space $X$ can be classified into three mutually exclusive types: elliptic, parabolic, and hyperbolic elements. The notion of hyperbolic elements plays a crucial role in the sequel.
\begin{definition}
    An isometry $g$ on $X$ is called  \textit{hyperbolic} if  the \textit{translation length} $\tau(g)$ of $g$ defined by 
    $$\tau(g):=\inf_{x\in X} d(x, gx)$$
    is strictly positive. 
\end{definition}
Note that the translation length is invariant under conjugacy. 

\begin{definition}
    A hyperbolic isometry $g$ has exactly two fixed points  on the boundary defined as $g^+:=\lim\limits_{n\to \infty} g^n o\in \partial X$ and $g^-:=\lim\limits_{n\to \infty} g^{-n} o \in \partial X$.
    
    By an \textit{axis} of $g$ we mean a geodesic line $[g^+, g^-]$ between the two fixed points.
\end{definition}

By Lemma~\ref{MorseLemma}, there exists a constant $D_0>0$ depending only on $\delta$ such that the Hausdorff distance between any two axes of $g$ is bounded by $D_0$. The following result is well-known. We include a proof as we are unable to find appropriate reference.
\begin{lemma}\label{translationDistance}
    There exists a constant $D_1>0$ such that 
    $$
    \tau(g)\le d(x, gx)\le \tau(g) + D_1
    $$
    for any hyperbolic element $g\in G$ and any point $x$ on some axis of $g$.
\end{lemma}

\begin{proof}
    Suppose $\gamma$ is an axis of $g$. 
    By definition, we have $d(x, gx)\ge \tau(g)$ for any $x\in \gamma$, and there exists $y\in X$ such that $\tau(g)\le d(y, gy)\le \tau(g)+\delta$.
    Then for any $n\in \mathbb{Z}$ and any $y^\prime\in g^n[y, gy]$, we have $gy^\prime\in g^{n+1}[y, gy]$, and 
    \begin{align}
        \tau(g)
        \le d(y^\prime, gy^\prime)
        \le d(y^\prime, g^{n+1}y)+d(gy^\prime, g^{n+1}y)
        =d(y, gy)
        \le \tau(g)+\delta.\notag
    \end{align}
    
    If $d([y, gy], \gamma)\le 8\delta$, then there exist $z\in \gamma$ and $y^\prime\in [y, gy]$ such that 
    $$d(z, gz)\le d(z, y^\prime)+d(y^\prime, gy^\prime)+d(gy^\prime, gz)\le \tau(g)+17\delta.$$
    
    If $d([y, gy], \gamma)>8\delta$, then $\operatorname{diam}(\pi_{\gamma}(y)\cup \pi_{\gamma}(gy))\le 8\delta$ by Lemma~\ref{geodesicContracting} from below.
    Let $z$ be a point in $\pi_{\gamma}(y)$. 
    Then $gz\in \pi_{g\gamma}(gy)$.
    By Lemma~\ref{finiteHausdorffDistanceProjection} below, we have $d_{\mathrm{Haus}}(\pi_{g\gamma}(gy), \pi_{\gamma}(gy))\le 2D_0+100\delta$.
    Hence $$d(z, gz)\le \operatorname{diam}(\pi_{\gamma}(y)\cup \pi_{\gamma}(gy))+d_{\mathrm{Haus}}(\pi_{g\gamma}(gy), \pi_{\gamma}(gy))\le 2D_0+108\delta.$$

    In summary, there exists a point $z\in \gamma$ such that $d(z, gz)\le \tau(g)+2D_0+200\delta$.

    Analogously, for any $n\in\mathbb{Z}$ and any $z^\prime\in g^n[z, gz]$, we have 
    $$\tau(g)\le d(z^\prime, gz^\prime)\le \tau(g)+2D_0+108\delta.$$

    For $n\in \mathbb{Z}$, choose $x_n\in \pi_{\gamma}(g^nz)$. 
    Then $d(x_n, g^nz)\le D_0$ and hence $\lim_{n\to +\infty} x_n = g^+, \lim_{n\to -\infty} x_n = g^-$.
    This implies that $\gamma=\bigcup_{n\in \mathbb{Z}}[x_n, x_{n+1}]_{\gamma}$.

    Therefore, given $x\in \gamma$, there exists $n\in \mathbb{Z}$ such that $x\in [x_{n}, x_{n+1}]_{\gamma}$.
    Since the quadrilateral $(x_n, g^n z, g^{n+1}z, x_{n+1})$ is $2\delta$-thin, for any $x\in [x_n, x_{n+1}]_{\gamma}$ we have 
    $$d(x, [x_n, g^nz]\cup [g^nz, g^{n+1}z]\cup[g^{n+1}z, x_{n+1}])\le 2\delta.$$

    If $d(x, [x_n, g^nz]\cup [x_{n+1}, g^{n+1}z])\le 2\delta$, then either $d(x, g^nz)\le D_0+2\delta$ or $d(x, g^{n+1}z)\le D_0+2\delta$.
    Without loss of generality, assume $d(x, g^nz)\le D_0+2\delta$.
    Then we have 
    $$d(x, gx)\le d(x, g^nz)+d(g^nz, g^{n+1}z)+d(g^{n+1}z, gx)\le \tau(g)+4D_0+112\delta.$$

    If $d(x, [g^{n}z, g^{n+1}z])\le 2\delta$, then there exists $z^\prime\in [g^{n}z, g^{n+1}z]$ such that $d(x, z^\prime)\le 2\delta$.
    Hence 
    $$d(x, gx)\le d(x, z)+d(z, gz)+d(gz, gx)\le \tau(g)+2D_0+112\delta. $$

    Thus, taking $D_1=4D_0+112\delta$ completes the proof.
\end{proof}

Until the end of this subsection, we discuss the strong contracting property of the shortest projection to quasigeodesics. This is also standard but less written in the literature, so we give some details at the convenience of the reader.
\begin{definition}
    For a closed subset $A \subset X$, the projection $\pi_{A}: X\to \mathcal{P}(A)$ to $A$ is defined as
$$
\pi_{A}(x):=\{y\in A: d(x, y)=d(x, A)\},
$$
where $\mathcal{P}(A)$ denotes the set of all subsets of $A$. 
\end{definition}
    Note that $\pi_{A}(x)$ is non-empty, thanks to the proper assumption on $X$.

\begin{lemma}\label{geodesicContracting}
    Let $\gamma$ be a geodesic line in $X$ and $x, y\in X$ such that $d(x^\prime, y^\prime)> 8\delta$ for some $x^\prime \in \pi_{\gamma}(x)$ and $y^\prime \in \pi_{\gamma}(y)$. 
    Then there exists a subpath $\gamma^\prime$ of $[x, y]$ such that $d_{\mathrm{Haus}}(\gamma^\prime, [x^\prime, y^\prime])\le 8\delta$, where $d_{\mathrm{Haus}}$ denotes the Hausdorff distance.

    Conversely, if $d([x, y], \gamma) > 8\delta$, then $\operatorname{diam}\bigl(\pi_{\gamma}([x, y])\bigr)\le 8\delta.$
\end{lemma}
\begin{proof}
    Let $z$ be the midpoint of $[x^\prime, y^\prime]$; thus $d(z, x^\prime)=d(z, y^\prime)> 4\delta$. 
    We claim that $d(z, [x, x^\prime])>2\delta$ and $d(z, [y, y^\prime])>2\delta$.

    Assume to the contrary that there exists $z^\prime\in [x, x^\prime]\cup [y, y^\prime]$ such that $d(z, z^\prime)\le 2\delta$.
    Without loss of generality, assume that $z^\prime\in [x, x^\prime]$.
    Then $d(z, x^\prime)\le d(z, z^\prime) + d(z^\prime, x^\prime)\le 2d(z, z^\prime)\le 4\delta$, a contradiction.
    
    Since the quadrilateral $(x, x^\prime, y^\prime, y)$ is $2\delta$-thin, there exist points in $[x, y]$ whose distance to $[x^\prime, y^\prime]$ is at most $2\delta$. 
    Let $p, q \in [x, y]$ be the points closest to $x$ and $y$, respectively, that lie in the $2\delta$-neighborhood of $[x^\prime, y^\prime]$.
    Then $d(p, [x, x^\prime])=d(q, [y, y^\prime])=2\delta$.

    Choose $p^\prime \in [x, x^\prime]$ and $q^\prime\in [y, y^\prime]$ such that $d(p, p^\prime)=d(q, q^\prime)=2\delta$.
    Then 
    $$d(p, x^\prime)\le d(p, p^\prime)+d(p^\prime, x^\prime)\le 2\delta + d(p^\prime, p)+d(p, [x^\prime, y^\prime])\le 6\delta,$$
    and similarly $d(q, y^\prime)\le 6\delta$.
    Since the (ideal) quadrilateral $(p, x^\prime, y^\prime, q)$ is $2\delta$-thin, we have $d_{\mathrm{Haus}}([p, q], [x^\prime, y^\prime])\le 8\delta$.
\end{proof}

The projection to a geodesic line $\gamma$ can be extended to the boundary as follows: $$\pi_{\gamma}:\overline{X}\setminus\{\gamma_-, \gamma_+\}\to \gamma$$
Indeed, suppose $\{x_n\}$ is a sequence in $X$ which converges to $\xi\in \partial X$ and $y_n\in \pi_{\gamma}(x_n)$.
Then any limit point of $\{y_n\}$ lies in $\pi_{\gamma}(\xi)$.
Note that $d([x_n, x_m], \gamma)>8\delta$ for sufficiently large $m, n>0$.
By Lemma~\ref{geodesicContracting} we obtain that $d(y_n, y_m)\le 8\delta$.
Therefore, the extension map $\pi_{\gamma}: \overline{X}\setminus\{\gamma_-, \gamma_+\}\to \gamma$ is $8\delta$-coarsely continuous.
Hence Lemma~\ref{geodesicContracting} can be generalized as follows.

\begin{corollary}\label{projectionFellowTravel}
    Let $\gamma$ be a geodesic line in $X$ and $x, y\in \overline{X}$ such that $d(x^\prime, y^\prime)> 24\delta$ for some $x^\prime \in \pi_{\gamma}(x)$ and $y^\prime \in \pi_{\gamma}(y)$. 
    Then there exists a subpath $\gamma^\prime$ of $[x, y]$ such that $d_{\mathrm{Haus}}(\gamma^\prime, [x^\prime, y^\prime])\le 24\delta$, where $d_{\mathrm{Haus}}$ denotes the Hausdorff distance.

    Conversely, if $d([x, y], \gamma) > 24\delta$, then $\operatorname{diam}\bigl(\pi_{\gamma}([x, y])\bigr)\le 24\delta.$
\end{corollary}

\begin{lemma}\label{finiteHausdorffDistanceProjection}
    Let $\gamma_1, \gamma_2$ be two geodesics in a $\delta$-hyperbolic space $X$ such that 
    $$d_{\mathrm{Haus}}(\gamma_1, \gamma_2)\le D$$
    for some $D>0$.
    Then $d_{\mathrm{Haus}}(\pi_{\gamma_1}(x), \pi_{\gamma_2}(x))\le 2D+100\delta$ for any $x\in \overline{X}$.
\end{lemma}
\begin{proof}
    It suffices to prove that $d(y_1, y_2)\le 2D+100\delta$ for $x\in X$, $y_1\in \pi_{\gamma_1}(x)$ and $y_2\in \pi_{\gamma_2}(x)$.
    By the definition of the Hausdorff distance, we have $d(x, y_1)\le d(x, \gamma_2)+D$.
    Let $z\in \pi_{\gamma_2}(y_1)$.
    If $d(z, y_2)\le D+100\delta$, then $d(y_1, y_2)\le 2D+100\delta$.
    Now assume that $d(z, y_2)>D+100\delta$.
    
    By Corollary~\ref{projectionFellowTravel}, there exists a subpath $[y^\prime, x^\prime]$ of $[y_1, x]$ such that 
    $$d_{\mathrm{Haus}}([y^\prime, x^\prime], [z, y_2]_{\gamma_2})\le 24\delta.$$
    Therefore,
    \begin{align}
    &d(x, \gamma_2)\notag\\
    \le &d(x, x^\prime)+24\delta\notag\\
    \le &d(x, y_1)-d(x^\prime, y^\prime)+24\delta \notag\\
    \le &d(x, \gamma_2) + D - (d(z, y_2)-24\delta - 24\delta) + 24\delta\notag\\
    < &d(x, \gamma_2) - 28\delta,\notag
    \end{align}
    which is a contradiction. 
\end{proof}

\subsection{Geometrically finite actions}
For the remainder of this section, we assume that a group $G$ acts properly by isometries on a $\delta$-hyperbolic space $X$, where $\delta\ge 0$, and fix a basepoint $o \in X$.

The action of $G$ on $X$ extends naturally to an action on the Gromov boundary $\partial X$. The induced boundary action is always a convergence group action. Furthermore, an element $g \in G$ is elliptic, parabolic, or hyperbolic with respect to the isometry classification on $X$ if and only if it is, respectively, elliptic, parabolic, or hyperbolic with respect to the induced convergence action on $\partial X$.

The geometry of $X$ provides a convenient characterization of conical limit points. A point $\xi \in \partial X$ is conical if and only if, for some (and hence any) geodesic ray $\gamma$ with $\gamma(\infty) = \xi$, there exists a sequence $t_n \to \infty$ and a constant $R > 0$ such that $d(\gamma(t_n), Go) \le R$.

\begin{definition}
    The \textit{weak-convex hull} $\mathrm{CH}(\Lambda G)$ of $\Lambda G$ is the union of all bi-infinite geodesics with two endpoints in $\Lambda G$. 
\end{definition}

By the $\delta$-thin triangle property, one sees that $\mathrm{CH}(\Lambda G)$ is a $G$-invariant \textit{uniformly quasi-convex} subset of $X$. Namely, there exists a constant $D$ (depending only on $\delta$) such that any geodesic with endpoints in $\mathrm{CH}(\Lambda G)$ is contained in the $D$-neighborhood of $\mathrm{CH}(\Lambda G)$. Further, the topological boundary of $\mathrm{CH}(\Lambda G)$ in $\partial X$ coincides with the limit set $\Lambda G$. 

Following \cite{bowditch2012relatively}, we say that a family of quasi-convex subsets in $X$ is \textit{sufficiently separated} if it is $r$-separated for a sufficiently large constant $r$ depending on $\delta$.  The following notion  was based on the work of Marden \cite{Marden} for Kleinian groups, and further developed by Bowditch in various contexts \cite{Bow3,Bow4,bowditch2012relatively}.
\begin{definition}\label{defn:geomfiniteonhypspace}
The action of $G$ on $X$ is called \textit{geometrically finite} if there exists a sufficiently separated  $G$-invariant family $\{B_p : p \in \Pi\}$ of horoballs in $X$, where each horoball $B_p$ is centered at a point $p$ in a $G$-invariant set $\Pi \subset \partial X$, and the action of $G$ on the complement $\mathrm{CH}(\Lambda G) \setminus \bigcup_{p \in \Pi} B_p$ is cocompact.    
\end{definition}

It is a fundamental fact that for proper actions on hyperbolic spaces, the geometric definition of geometric finiteness  is equivalent to the dynamical definition in terms of the boundary action in Definition \ref{defn:geomfiniteconvergence}. We record this equivalence in the following lemma.

\begin{lemma}[\cite{bowditch2012relatively}, Proposition 6.13]\label{geomFinite}
    The proper action of $G$ on $X$ is geometrically finite if and only if the action of $G$ on $\partial X$ is geometrically finite.
\end{lemma}

\section{Escaping Closed Geodesics}
In this section, we assume that a group $G$ acts properly by isometries on a proper $\delta$-hyperbolic space $X$ for some $\delta\ge 0$. Moreover, we assume that the action is non-elementary. We fix a basepoint $o\in X$. Recall that $d_{\mathrm{Haus}}(\gamma_1, \gamma_2)\le D_0$ for any hyperbolic element $g\in G$ and any axes $\gamma_1, \gamma_2$ of $g$.

The goal of this section is to establish Theorem~\ref{geometricallyFiniteNearOrbit}.

The nontrivial direction is to show that, if there is no escaping sequence of hyperbolic elements, then the action is geometrically finite. Accordingly, we impose the following assumption up to Subsection~\ref{subsec:proofTheorem1}, where this implication, together with the converse direction in Theorem~\ref{geometricallyFiniteNearOrbit}, will be proved using the preparatory results established below.

\medskip
\noindent{\textbf{Non-Escaping Hyperbolic Elements (NEHE) Assumption}} 
\\
There exists $R_0>0$ such that for every hyperbolic element $g \in G$, some axis $\gamma$ of $g$ satisfies $d(\gamma, G o) \le R_0$.

\medskip
Since the action is non-elementary, there are infinitely many conjugacy classes of  hyperbolic elements. An immediate consequence of the \textbf{NEHE Assumption} is the following finiteness statement about conjugacy classes.
\begin{lemma}\label{stableTranslationLength}
    For any $r > 0$, there are only finitely many conjugacy classes of hyperbolic elements in $G$ with translation length less than $r$.
\end{lemma}
\begin{proof}
    Let $r > 0$, and let $g \in G$ be a hyperbolic element with $\tau(g) < r$. By our assumption, there exists $h \in G$ and an axis $\gamma$ of $g$ such that $d(h o, \gamma) \le R_0$. Thus, we can choose a point $x \in \gamma$ satisfying $d(h o, x) \le R_0$. Applying the triangle inequality and using Lemma~\ref{translationDistance}, we estimate:
    $$
    d(h o, g h o) \le d(h o, x) + d(x, g x) + d(g x, g h o) \le R_0 + \tau(g) + D_1 + R_0 < 2R_0 + D_1 + r.
    $$
    This yields $d(o, h^{-1} g h \, o) = d(h o, g h o) < 2R_0 + D_1 + r$.
    Since the action of $G$ on $X$ is proper, the set $\{ k \in G : d(o, k o) < 2R_0 + D_1 + r \}$ is finite. Therefore, there is an element $h^{-1}gh$ which is conjugate to $g$ and belongs to this finite set. Consequently, there are only finitely many conjugacy classes of such elements.
\end{proof}

\subsection{The main proposition}
Assume that the non-conical limit set $\Lambda^{nc}G$ is non-empty; otherwise there is nothing to prove.  
Maintaining \textbf{NEHE Assumption}, our next goal is to prove the following key estimate, which describes the geometry of lines connecting a non-conical limit point to its translates.

\begin{proposition}\label{lineNearOrbit}
    For any $\xi\in \Lambda^{nc} G$, there exists a constant $R_{\xi}$ depending on $\xi$ such that $d([\xi, g\xi], G o)\le R_{\xi}$ for every $g\in G$ with $g\xi\ne \xi$.
\end{proposition}

To establish Proposition \ref{lineNearOrbit}, we first prove several auxiliary lemmas. In the following lemmas, since the fixed points of every hyperbolic element $g\in G$ are conical, we have $g\xi\ne \xi$ for any non-conical point $\xi$.

\begin{lemma}\label{translationLengthLong}
    There exists a constant $R_1=R_1(R_0, \delta)$ such that $d([\xi, g\xi], G o)\le R_1$ for any $\xi\in \Lambda^{\mathrm{nc}} G$ and for any hyperbolic element $g\in G$ with $\tau(g)\ge 2D_0+ 200\delta$.
\end{lemma}
\begin{proof}
    For any hyperbolic element $g\in G$ with $\tau(g)\ge 2D_0+ 200\delta$, our \textbf{NEHE Assumption}  ensures the existence of an element $h\in G$, an axis $\gamma$ of $g$ and a point $x\in \gamma$ such that $d(x, h o)\le R_0$. 
    Then for any $n\in \mathbb{Z}$, the points $g^n h o\in G o$ and $g^n x\in g^n\gamma$ also satisfy $d(g^n h o, g^n x)\le R_0$ and $g^n\gamma$ is also an axis of $g$. 
    Hence $d_{\mathrm{Haus}}(g^n\gamma, \gamma)\le D_0$ and $d(g^n ho, \gamma)\le R_0+D_0$ for every $n\in \mathbb{Z}$.
    
    For each $n\in \mathbb{Z}$, fix $x_n\in \pi_{\gamma}(g^n ho)$.
    Then $gx_n\in \pi_{g\gamma}(g^{n+1}ho)$.
    By Lemma~\ref{finiteHausdorffDistanceProjection}, we have $d(gx_n, x_{n+1})\le 2D_0+100\delta$.
    Hence, by Lemma~\ref{translationDistance},
    \begin{align}
        d(x_n, x_{n+1})\le d(x_n, gx_n)+d(gx_n, x_{n+1})\le \tau(g)+D_1+2D_0+100\delta.\label{findOrbitPoints}
    \end{align}
    
    Choose points $y\in \pi_{\gamma}(\xi)$ and $z\in \pi_{\gamma}(g\xi)$. 
    Then we have $g y\in \pi_{g \gamma}(g \xi)$. 
    Since $d_{\mathrm{Haus}}(\gamma, g\gamma)\le D_0$, Lemma~\ref{finiteHausdorffDistanceProjection} implies $d(z, gy)\le 2D_0+100\delta$.
    Hence 
    $$d(y, z)\ge d(y, gy)-d(gy, z)\ge \tau(g) - 2D_0-100\delta \ge 100\delta.$$
    By Corollary~\ref{projectionFellowTravel}, there exists a subpath of the bi-infinite line $[\xi, g\xi]$ that lies within Hausdorff distance $24\delta$ of the subpath $[y, z]$.

    On the other hand, by inequality~\eqref{findOrbitPoints} there exists $n\in \mathbb{Z}$ such that 
    $$d([y, z], x_n)\le \tau(g)+D_1+2D_0+100\delta - d(y, z)\le D_1+4D_0+200\delta.$$ 
    Therefore, for every such $g$, we have 
    \begin{align}
    &d([\xi, g\xi], G o)\notag\\
    \le& d_{\mathrm{Haus}}([\xi, g\xi], [y, z]) + d([y, z], x_n)+d(x_n, g^nho)\notag\\
    \le& 24\delta + D_1+4D_0+200\delta + R_0+D_0\notag\\
    =& R_0+224\delta+5D_0+D_1. \notag
    \end{align}
    
    Setting $R_1 = R_0+224\delta+5D_0+D_1$ completes the proof.
\end{proof}

For hyperbolic elements with short translation length, the bound depends on the specific conjugacy class. The next lemma provides a uniform bound for all elements conjugate to a fixed hyperbolic element.

\begin{lemma}\label{fixHyperbolicElement}
    Let $[g]$ denote the conjugacy class of a hyperbolic element $g$ in $G$. Then there exists a constant $R_{[g]}\ge 0$ such that $d([\xi, g \xi], G o)\le R_{[g]}$ for any $\xi\in \Lambda^{\mathrm{nc}} G$ and for any element $g\in [g]$.
\end{lemma}
\begin{proof}
    Fix a hyperbolic element $g\in G$.
    Note that the action of the subgroup $\langle g\rangle$ on $\partial X\setminus\{g^+, g^-\}$ is proper and cocompact.
    Choose a compact fundamental domain $K$ for this action.
    Let $\varepsilon>0$ be sufficiently small so that the visual metric $d_{o, \varepsilon}$ is well-defined.
    Since the action is proper and $K$ is compact, the distance 
    $d_{o, \varepsilon}(\xi, g\xi)$ is bounded below by a positive constant $c>0$ on $K$.
    By the estimate for the visual metric (\ref{eqn: visual metric}), we have
    $$d([\xi, g\xi], o)\le -\frac{1}{\varepsilon}\log d_{o, \varepsilon}(\xi, g\xi)\le -\frac{1}{\varepsilon}\log c:=R_{g}.$$
    Thus, $d([\xi, g\xi], Go)\le R_g$ for any $\xi\in K$.
    Since the function $\xi\mapsto d([\xi, g\xi], Go)$ is $\langle g\rangle$-invariant, the above estimate holds for every $\xi\in \partial X\setminus\{g^+, g^-\}$.

    Now, let $g^\prime = h^{-1} g h$ be an element conjugate to $g$ for some $h\in G$. For $\xi\in \Lambda^{\mathrm{nc}} G$, we have
    $$
    d([\xi, g^\prime \xi], G o)=d([\xi, h^{-1} g h\xi], G o)=d([h\xi, g h\xi], G o)\le R_g,
    $$
    where the last inequality follows from the previous paragraph because $h\xi \in \Lambda^{nc}G$.
    Hence, we can set $R_{[g]} = R_g$, where $[g]$ denotes the conjugacy class containing $g$.
    This completes the proof.
\end{proof}

Combining the previous two lemmas, we obtain a uniform bound valid for all hyperbolic elements.

\begin{lemma}\label{translationNearOrbit}
    There exists a constant $R_2\ge R_0$ such that $d([\xi, g\xi], G o)\le R_2$ for any $\xi\in \Lambda^{\mathrm{nc}} G$ and for any hyperbolic element $g\in G$.
\end{lemma}
\begin{proof}
    Let $\mathcal{C}$ denote the collection of conjugacy classes of hyperbolic elements with stable translation length less than $2D_0+200\delta$, which is a finite set by Lemma~\ref{stableTranslationLength}.
    Let $R_c$ be the constant given by Lemma~\ref{fixHyperbolicElement} for each conjugacy class $c \in \mathcal{C}$. For hyperbolic elements with $\tau(g) \ge 2D_0+200\delta$, let $R_1$ be supplied by Lemma~\ref{translationLengthLong}.
    Defining $R_2 = \max\left( \{R_c : c \in \mathcal{C} \} \cup \{R_1\} \right)$ yields the desired constant.
\end{proof}

We are now ready to prove the main proposition.

\begin{figure}[htbp]
\centering
\begin{tikzpicture}[
    scale=1,
    point/.style={circle, fill=black, inner sep=1.2pt},
    boundarypoint/.style={circle, fill=red!70, inner sep=1.5pt},
    neighborhood/.style={line width=1.2pt},
    font=\small
]

% 定义圆弧的半径和角度范围
\def\radius{5}
\def\startangle{50}
\def\endangle{130}

% 画圆弧表示边界 ∂X
\draw[thick] (\startangle:\radius) arc (\startangle:\endangle:\radius);
\node at (40:5) {$\partial X$};

% 标记三个边界点：点保持在圆弧上（半径5）
\coordinate (xi) at (110:\radius);
\coordinate (gNxi) at (80:\radius);
\coordinate (eta) at (65:\radius);

% 标注点：只调整标签位置，不移动点
% 使用 label={[shift={...}]...} 来微调标签位置
\node[boundarypoint, label={[shift={(0,-0.6)}]below:$\xi$}] at (xi) {};
\node[boundarypoint, label={[shift={(0,-0.4)}]below:$g_N\xi$}] at (gNxi) {};
\node[boundarypoint, label=above:$\eta$] at (eta) {};

% 计算旋转角度
\def\rotangle#1{%
    \pgfmathparse{#1+90}%
    \pgfmathresult%
}

% ξ 的邻域 U
\pgfmathsetmacro{\xitheta}{110}
\pgfmathsetmacro{\xirot}{\xitheta+90}
\draw[neighborhood, blue!80, rotate around={\xirot:(xi)}] 
    (xi) ellipse (1.6 and 0.6);
\node[blue!80] at ($(xi)+(0.8,1.2)$) {$U$};

% ξ 的邻域 V
\draw[neighborhood, red!80, rotate around={\xirot:(xi)}] 
    (xi) ellipse (1.0 and 0.4);
\node[red!80] at ($(xi)+(-0.5,0.7)$) {$V$};

% g_Nξ 的邻域 g_N U
\pgfmathsetmacro{\gNtheta}{80}
\pgfmathsetmacro{\gNrot}{\gNtheta+90}
\draw[neighborhood, green!70!black, rotate around={\gNrot:(gNxi)}] 
    (gNxi) ellipse (0.9 and 0.4);
\node[green!70!black] at ($(gNxi)+(0.4,0.7)$) {$g_NU$};

% g_n g_N U
\draw[neighborhood, orange!80, rotate around={\xirot:(xi)}] 
    (xi) ellipse (0.6 and 0.3);
% 将g_n g_NU的标注向右移动更多
\node[orange!80] at ($(xi)+(1.2,-0.3)$) {$g_n g_NU$};

\end{tikzpicture}
\caption{Illustrating the hyperbolic elements $g_ng_N$ in the proof of Proposition~\ref{lineNearOrbit}.}
\label{fig:neighborhoods}
\end{figure}

\begin{proof}[Proof of Proposition \ref{lineNearOrbit}]
    Suppose, for the sake of contradiction, that for each $n\in \mathbb{N}$ there exists an element $g_n\in G$ with $g_n\xi\ne \xi$ such that $d([\xi, g_n\xi], G o)\ge n$.
    In particular, $d([\xi, g_n\xi], o)\ge n$, which implies that the visual distance between $g_n\xi$ and $\xi$ tends to $0$.
    Hence $\lim\limits_{n\to\infty} g_n\xi = \xi$.
    
    By the convergence property, after passing to a subsequence, we may assume that $g_n \to \zeta$ locally uniformly on $\overline{X}\setminus \{\eta\}$ for some $\eta, \zeta\in \partial X$. We claim that $\zeta=\xi$. Indeed, if $\zeta\ne \xi$, then every subsequence $\{g_n\}$ would converge to some $\zeta\ne \xi$ locally uniformly on $\overline{X}\setminus \{\xi\}$. This would imply that $\xi$ is a conical point, a contradiction.
    
    Fix an integer $N$ such that $N > R_2 + 2\delta$ and $g_N\xi \ne \eta$. We claim that the product $g_n g_N$ is a hyperbolic element for all sufficiently large $n$.

    Choose neighborhoods $U$ and $V$ of $\xi$ in $\partial X$ such that $\xi \in V \subset \overline{V} \subsetneq U$ and $\eta \notin g_N \overline{U}$. Such neighborhoods exist because $\partial X$ is a compact Hausdorff space with no isolated points and $g_N \xi \ne \eta$.
    Since $g_n x \to \xi$ locally uniformly on $\overline{X}\setminus \{\eta\}$, for sufficiently large $n$ we have $g_n (g_N \overline{U}) \subseteq V \subsetneq U$. See Figure \ref{fig:neighborhoods}. By Lemma~\ref{producingHyperbolicElements}, this implies that $g_n g_N$ is hyperbolic, proving the claim.

    Now let $n$ be large enough so that $g_n g_N$ is hyperbolic and $n > R_2 + 2\delta$.
    By Lemma~\ref{translationNearOrbit}, we have $d([\xi, g_n g_N \xi], G o) \le R_2$.
    Since the ideal triangle $\triangle(\xi, g_n\xi, g_ng_N\xi)$ is $2\delta$-thin, the geodesic $[\xi, g_n g_N \xi]$ lies within $2\delta$ of $[\xi, g_n \xi] \cup [g_n \xi, g_n g_N \xi]$. See Figure \ref{fig:ideal_triangle}.
    Therefore,
    $$
    d([\xi, g_n \xi], G o) \le R_2 + 2\delta \quad \text{or} \quad d([g_n \xi, g_n g_N \xi], G o) \le R_2 + 2\delta.
    $$
    
    However, by our choice of $n$ and $N$, we have $d([\xi, g_n\xi], Go) \ge n > R_2+2\delta$ and $d([\xi, g_{N}\xi], Go) \ge N > R_2+2\delta$.
    Hence $d([g_n\xi, g_ng_N\xi], Go) > R_2+2\delta$, which is a contradiction.
\end{proof}

% 需要在导言区添加以下包：
% \usepackage{tikz}
% \usetikzlibrary{calc}

\begin{figure}[htbp]
\centering
\begin{tikzpicture}[
    scale=1.2,
    point/.style={circle, fill=black, inner sep=1.2pt},
    boundarypoint/.style={circle, fill=red!70, inner sep=1.5pt},
    orbitpoint/.style={circle, fill=blue!70, inner sep=1.2pt},
    geodesic/.style={thick, black},
    font=\small
]

% 定义理想三角形的三个顶点
\coordinate (xi) at (90:3);
\coordinate (gnxi) at (210:3);
\coordinate (gngNxi) at (330:3);

% 标注三个边界点
\node[boundarypoint, label=above:$\xi$] at (xi) {};
\node[boundarypoint, label=below left:$g_n\xi$] at (gnxi) {};
\node[boundarypoint, label=below right:$g_n g_N \xi$] at (gngNxi) {};

% 计算每条边向内凸的控制点：向三角形中心（原点）偏移更多（80%）
% 边 [ξ, g_n g_N ξ] 的控制点：取中点并向中心偏移80%
\coordinate (mid1) at ($(xi)!0.5!(gngNxi)$);
\coordinate (ctrl1) at ($(mid1)!0.8!(0,0)$); % 80% 向中心偏移

% 边 [ξ, g_n ξ] 的控制点
\coordinate (mid2) at ($(xi)!0.5!(gnxi)$);
\coordinate (ctrl2) at ($(mid2)!0.8!(0,0)$);

% 边 [g_n ξ, g_n g_N ξ] 的控制点
\coordinate (mid3) at ($(gnxi)!0.5!(gngNxi)$);
\coordinate (ctrl3) at ($(mid3)!0.8!(0,0)$);

% 画理想三角形的三条边：使用向内凸的圆弧
\draw[geodesic] (xi) .. controls (ctrl1) .. (gngNxi);
\draw[geodesic] (xi) .. controls (ctrl2) .. (gnxi);
\draw[geodesic] (gnxi) .. controls (ctrl3) .. (gngNxi);

% 在边 [ξ, g_n g_N ξ] 附近画一个点 Go
\coordinate (Go) at ($(mid1)!0.3!(0,0)$);
\node[orbitpoint] at (Go) {};
\node[above right] at (Go) {$ho\in Go$};

\end{tikzpicture}
\caption{Ideal triangle and a nearby orbit point.}
\label{fig:ideal_triangle}
\end{figure}

\subsection{Identifying the cusp regions}\label{horoballConstruction}

We now introduce the key geometric sets that will be used to analyze the dynamics near a non-conical limit point. These sets are essential for establishing the geometric finiteness criterion in the sufficiency direction of Theorem~\ref{geometricallyFiniteNearOrbit}. 

Our strategy is to construct, for each non-conical limit point $\xi$, a geometric object $B_{\xi,R}$, which plays the role of an invariant horoball. We will show that for sufficiently large $R$, the translates of these sets are pairwise disjoint, and that their truncations $S_{\xi,R} = B_{\xi,R} \cap Y_R$ have compact quotient under the stabilizer $G_\xi$. This will prove that every $\xi \in \Lambda^{nc}G$ is a bounded parabolic point, thereby completing the argument.

Recall the \textbf{NEHE Assumption} that the constant $R_0$ satisfies $d(\gamma, G o) \le R_0$ for every hyperbolic element $g \in G$ and some axis $\gamma$ of $g$. 

For $R \ge R_0$ and $\xi \in \Lambda^{nc} G$, define $Y_R := \{x \in X : d(x, Go) \le R\}$ and 
\[
B_{\xi, R} := \left\{
x \in X \;\middle|\;
\begin{aligned}
&\text{there exists a geodesic ray } \gamma = [x, \xi] \\
&\text{such that } d(\gamma, Go) \ge R
\end{aligned}
\right\}.
\]
 
The following lemma justifies the equivariance of this construction.

\begin{lemma}\label{equivariant}
    For $g \in G$, $\xi \in \Lambda^{nc} G$, and $R \ge R_0$, we have $g B_{\xi, R} = B_{g\xi, R}$.
\end{lemma}
\begin{proof}
    Observe that $x \in B_{\xi, R}$ if and only if there exists a geodesic ray $\gamma = [x, \xi]$ lying outside $Y_R$. 
    This is equivalent to that $g\gamma = [gx, g\xi]$ lies outside $g Y_R = Y_R$, which means $gx \in B_{g\xi, R}$.
\end{proof}

The sets $B_{\xi, R}$ enjoy a crucial dynamical property: for a fixed non-conical limit point $\xi$, their translates under the group action become pairwise disjoint once $R$ is sufficiently large. This is precisely the content of the next lemma.

Let $R_\xi$ be the constant provided by Proposition~\ref{lineNearOrbit}.
\begin{lemma}\label{horoballDisjoint}
    For any $\xi\in \Lambda^{nc} G$ and for any $g\in G$ with $g\xi \ne \xi$, we have $B_{\xi, R} \cap B_{g\xi, R} = \varnothing$ for all $R > R_\xi + 2\delta$.
\end{lemma}
\begin{proof}
    Suppose, for contradiction, that there exist some $R > R_\xi + 2\delta$ and a point $x \in B_{\xi, R} \cap B_{g\xi, R}$.
    By definition of $B_{\xi, R}$ and $B_{g\xi, R}$, there exist geodesic rays $[x, \xi]$ and $[x, g\xi]$ such that $d([x, \xi], G o) \ge R$ and $d([x, g\xi], G o) \ge R$.
    
    By Proposition~\ref{lineNearOrbit}, we have $d([\xi, g\xi], G o) \le R_\xi$.
    Consider the ideal triangle $\triangle(\xi, g\xi, x)$, which is $2\delta$-thin.
    Therefore, any point on the side $[\xi, g\xi]$ lies within distance $2\delta$ of the union of the other two sides $[x,\xi]$ and $[x,g\xi]$.
    
    Let $p \in [\xi, g\xi]$ be a point such that $d(p, G o) = d([\xi, g\xi], G o) \le R_\xi$.
    Then there exists a point $q$ on either $[x,\xi]$ or $[x,g\xi]$ with $d(p,q) \le 2\delta$.
    Consequently,
    $$
    d(q, G o) \le d(p, G o) + 2\delta \le R_\xi + 2\delta  < R.
    $$
    Hence, the side containing $q$ satisfies $d([x, \xi], G o) < R$ or $d([x, g\xi], G o) < R$.
    This contradicts the assumption that both $d([x, \xi], G o) \ge R$ and $d([x, g\xi], G o) \ge R$.
\end{proof}

\begin{lemma}\label{eventuallyHoroball}
    Let $\xi\in \Lambda^{nc} G$ and $R\ge R_0$. Then any geodesic ray $\gamma$ with $\gamma(\infty)=\xi$ is eventually contained in $B_{\xi, R}$; that is, there exists $a\ge0$ such that $\gamma([a, \infty))\subset B_{\xi, R}$.
\end{lemma}
\begin{proof}
    Suppose, for contradiction, that there exists a geodesic ray $\gamma$ with $\gamma(\infty)=\xi$ that fails to be eventually contained in $B_{\xi, R}$. 
    Then for every $n\in \mathbb{N}$ the intersection $\gamma([n, \infty))\setminus B_{\xi, R}$ is nonempty. 
    Thus, we can choose a sequence $\{t_n\}$ with $t_n \ge n$ and $\gamma(t_n)\notin B_{\xi, R}$ for all $n$.
    
    Since $\gamma(t_n) \notin B_{\xi, R}$, the definition of $B_{\xi, R}$ yields 
    $$d(\gamma([t_n, \infty)), Go)<R.$$
    Hence, $\xi$ is a conical limit point, contradicting the assumption that $\xi \in \Lambda^{nc} G$.
\end{proof}
In particular, $B_{\xi, R}$ is non-empty and unbounded for any $R\ge R_0$.

\begin{lemma}\label{horoballClosed}
    Let $\xi\in \Lambda^{nc}G$ and $R\ge R_\xi + 2\delta$. Then $B_{\xi, R}$ is a closed subset in $X$.
\end{lemma}
\begin{proof}
    Let $\{x_n\}$ be a sequence of points in $B_{\xi, R}$ converging to a point $x\in X$.
    Then there exists a sequence $\{\gamma_n=[x_n, \xi]\}$ of geodesic rays such that $d(\gamma_n, Go)\ge R$.
    By the Arzelà–Ascoli lemma, there exists a subsequence $\{n_i\}$ such that $\gamma_{n_i}$ converges locally uniformly to a geodesic ray $\gamma=[x, \xi]$.
    Then $d(\gamma, Go)\ge R$, which implies $x\in B_{\xi, R}$. Thus, $B_{\xi, R}$ is closed.
\end{proof}

For any non-conical limit point $\xi\in \Lambda^{nc} G$, let $S_{\xi, R}=B_{\xi, R}\cap Y_{R}$. By Lemma~\ref{horoballClosed}, $S_{\xi, R}$ is a closed subset of $Y_R$, and by Lemma~\ref{equivariant} we have $gS_{\xi, R}=S_{g\xi, R}$.

\begin{corollary}\label{disjoint}
    For any $\xi\in \Lambda^{nc} G$, $g\in G$ with $\xi\ne g\xi$, and for any $R > R_\xi + 2\delta$, we have $S_{\xi, R} \cap S_{g\xi, R} = \varnothing$.
\end{corollary}
\begin{proof}
    This follows directly from Lemma~\ref{horoballDisjoint}.
\end{proof}

\begin{lemma}\label{boundaryInCocompact}
    For every $\xi\in \Lambda^{nc} G$ and $R> R_\xi+2\delta$, the set $S_{\xi, R}$ is an unbounded closed subset of $Y_R$. Moreover, $\xi$ belongs to the limit set of $S_{\xi, R}$.
\end{lemma}
\begin{proof}
    Consider the orbit $G\xi$ of $\xi$ in $\Lambda G$.
    Since the action of $G$ on $\Lambda G$ is topologically minimal, the orbit $G\xi$ is dense in $\Lambda G$.
    Hence there exists a sequence $\{g_n\}$ in $G$ such that $g_n\xi\ne \xi$ and $g_n\xi\to \xi$ in $\Lambda G$.
    For each $n$, choose a geodesic $\gamma_n=[g_n\xi, \xi]$.
    
    By Lemma~\ref{eventuallyHoroball}, each geodesic $\gamma_n$ is eventually contained in $B_{\xi, R}$.
    By Proposition~\ref{lineNearOrbit}, there exists a point $x_n\in \gamma_n$ such that $d(x_n, Go)\le R_\xi<R$.
    Since the function $d(\cdot, Go)$ is continuous along $\gamma_n$, and $\gamma_n$ eventually lies in $B_{\xi, R}$ (where the distance to $Go$ is at least $R$), there exists a point $y_n$ on the subpath $[x_n, \xi]_{\gamma_n}$ such that $d(y_n, Go)=R$ and $d([y_n, \xi]_{\gamma_n}, Go)\ge R$. Thus, $y_n\in B_{\xi, R}\cap Y_R = S_{\xi, R}$.

    Because $g_n\xi\to \xi$, the visual distance between $g_n\xi$ and $\xi$ tends to zero, which implies
    $$
    \lim\limits_{n\to\infty} d(o, \gamma_n)=\infty.
    $$
    Consequently,
    $$
    \lim\limits_{n\to\infty} d(o, [y_n, \xi]_{\gamma_n})=\infty.
    $$
    This forces $y_n$ to converge to $\xi$ in $\overline{X}$; in particular, $\xi$ is a limit point of $S_{\xi, R}$.
\end{proof}

\subsection{Completion of the proof of Theorem~\ref{geometricallyFiniteNearOrbit}}\label{subsec:proofTheorem1}

With the preceding geometric constructions established, we now establish the two directions of the equivalence separately.
\\
    \paragraph{\textit{Necessity direction}.} This direction is well-known and we include a proof for completeness. Assume the action of $G$ on $X$ is geometrically finite. By Lemma~\ref{geomFinite}, there exists a $G$-invariant family of pairwise disjoint horoballs $\{B_p: p\in \Lambda^{nc} G\}$ such that the action of $G$ on the complement $\mathrm{CH}(\Lambda G)\setminus \bigcup_{p\in \Lambda^{nc} G}B_p$ is cocompact.
    
    It is a known fact that no bi-infinite geodesics are contained in a horoball. Since horoballs $\{B_p: p\in \Lambda^{nc} G\}$ are   pairwise disjoint, any axis of any hyperbolic element $g \in G$ must intersect the complement $X\setminus \bigcup_{p\in \Lambda^{nc} G}B_p$.
    Suppose the weak hull $\mathrm{CH}(\Lambda G)$ is $D$-quasiconvex for some $D>0$.
    Then every axis $\gamma$ of $g$ lies in the $D$-neighborhood of $\mathrm{CH}(\Lambda G)$.
    Hence $\gamma\cap N_{D}(\mathrm{CH}(\Lambda G))\setminus\bigcup_{p\in \Lambda^{nc} G}B_p\ne \varnothing$.
    
    Since the action on the complement $\mathrm{CH}(\Lambda G)\setminus \bigcup_{p\in \Lambda^{nc} G}B_p$ is cocompact, there exists a constant $R>0$ such that the orbit $G \cdot B(o, R)$ covers this set. Therefore, $d(\gamma, G o) \le R+D$ for every hyperbolic element $g\in G$ and every axis $\gamma$ of $g$, proving necessity with $R_0 = R+D$.
\\
    \paragraph{\textit{Sufficiency direction}.} Conversely, assume  there exists a constant $R_0>0$ such that $d(\gamma, G o)\le R_0$ for every hyperbolic element $g$ in $G$ and some axis $\gamma$ of $g$ (i.e. \textbf{NEHE assumption}). We must show that the action is geometrically finite. By Lemma~\ref{geomFinite}, it is enough to prove that every non-conical limit point is a bounded parabolic point.

    Fix a non-conical limit point $\xi \in \Lambda^{nc} G$ and choose $R > R_\xi + 2\delta$, where $R_\xi$ is the constant from Proposition~\ref{lineNearOrbit}.
    Consider the quotient map $$\pi: Y_R \to Y_R/G$$ induced by the proper action.
    The image $\pi(S_{\xi, R})$ is a closed subset of the compact space $Y_R/G$, hence $\pi(S_{\xi, R})$ is compact.
    
    For any $g \in G$ with $g\xi \ne \xi$, Lemma~\ref{equivariant} gives $g S_{\xi, R} = S_{g\xi, R}$, and Corollary~\ref{disjoint} ensures that $S_{\xi, R} \cap S_{g\xi, R} = \varnothing$.
    Consequently, the $\pi$-preimage of $\pi(S_{\xi, R})$ is precisely the $G_\xi$-orbit of $S_{\xi, R}$, inducing a homeomorphism $$\pi(S_{\xi, R}) \cong S_{\xi, R} / G_\xi$$ where $G_\xi$ denotes the stabilizer of $\xi$ in $G$.
    The compactness of $\pi(S_{\xi, R})$ thus implies that the action of $G_\xi$ on $S_{\xi, R}$ is cocompact.

    By Lemma~\ref{boundaryInCocompact}, $S_{\xi, R}$ is unbounded. Since a cocompact action on a non-compact space requires an infinite group of isometries, the stabilizer $G_\xi$ must be infinite.
    Moreover, as $\xi$ is non-conical, its limit set under the action of $G_\xi$ is the singleton $\Lambda G_\xi = \{\xi\}$. Indeed, if $\#\Lambda G_{\xi}\ge 2$, then $G_\xi$ would contain a hyperbolic element, contradicting that $\xi$ is non-conical. Hence, $\xi$ is a parabolic point.

    It remains to prove $\xi$ is \textit{bounded} parabolic. To that end, we construct a compact fundamental domain for the action of $G_\xi$ on $\partial X \setminus \{\xi\}$.
    First, let $K \subset S_{\xi, R}$ be a compact fundamental domain for the cocompact action of $G_\xi$ on $S_{\xi, R}$. % (such a $K$ exists by the argument above).
    Define a boundary subset $D \subset \partial X \setminus \{\xi\}$ by
    $$
    D := \{\eta \in \partial X \setminus \{\xi\} : \text{there exists a geodesic line } [\eta, \xi] \text{ intersecting } K \}.
    $$

    \textit{Claim:} The closure $\overline{D}$ of $D$ in $\partial X \setminus \{\xi\}$ is a compact fundamental domain for the action of $G_\xi$ on $\partial X \setminus \{\xi\}$.

    \begin{proof}[Proof of the Claim]
        Since the action of $G_\xi$ on $S_{\xi, R}$ is cocompact, the limit set of $S_{\xi, R}$ is $\Lambda S_{\xi, R} = \{\xi\}$. This implies that any geodesic line $\gamma$ with $\gamma(\infty) = \xi$ must intersect $S_{\xi, R}$. For any $\eta \in \Lambda G \setminus \{\xi\}$, the line $[\eta, \xi]$ therefore meets some $G_\xi$-translate of $K \subset S_{\xi, R}$. Thus, $G_\xi D = \Lambda G \setminus \{\xi\}$.

        It remains to show $\overline{D} \subset \Lambda G \setminus \{\xi\}$. It suffices to verify that $\xi$ is not an accumulation point of $D$. Let $\{\eta_n\} \subset D$ be a sequence converging to some point in $\partial X$. For each $n$, choose a geodesic line $\gamma_n = [\eta_n, \xi]$ and a point $x_n \in \gamma_n \cap K$. Since $K$ is compact, the distances $d(o, x_n)$ are uniformly bounded. In a $\delta$-hyperbolic space, this yields a uniform lower bound on the Gromov product $\langle  \eta_n, \xi \rangle_o$, and hence a positive lower bound on the visual distance between $\eta_n$ and $\xi$ with respect to any visual metric on $\partial X$. Therefore, $\lim\limits_{n\to\infty} \eta_n \neq \xi$, confirming $\overline{D} \subset \partial X \setminus \{\xi\}$ and its compactness.
    \end{proof}

    The existence of such a compact fundamental domain $\overline{D}$ means precisely that $\xi$ is a bounded parabolic point. Since $\xi \in \Lambda^{nc} G$ was arbitrary, every non-conical limit point is bounded parabolic. Applying Lemma~\ref{geomFinite} (or the standard dynamical characterization of geometric finiteness), we conclude that the action of $G$ on $\Lambda G$ is geometrically finite. The Sufficiency direction direction follows and Theorem \ref{geometricallyFiniteNearOrbit} is proved.
%\end{proof}

\section{Uncountably Many Non-conical Points}
This section is devoted to proving Theorem~\ref{geometricallyFiniteNonconicalCountable} based on Theorem~\ref{geometricallyFiniteNearOrbit}. 

Before turning to the proof, we note that the strategy originating in \cite{Bishop96}
(and further developed in \cite{KL19}) has by now become standard. 
However, at the level of generality considered here, a few additional arguments are needed.
In particular, one must verify that the geodesic rays produced by this method are genuinely
\emph{escaping}, so that their endpoints are non-conical.

In the case of pinched negatively curved manifolds, this is proved in \cite[Lemma~4.1]{KL19}
via delicate arguments that use manifold topology. 
By contrast, our proof below is purely metric: it avoids any manifold-specific topological input.

Throughout this section, let $G$ be a discrete group that admits a non-elementary, geometrically infinite, proper action on a proper $\delta$-hyperbolic space $X$ for some $\delta\ge 0$, and fix a basepoint $o\in X$. 
By Lemma~\ref{localQuasiGeodesic}, there exist constants $L>16\delta$, $\lambda \ge 1$, and $C\ge 0$ such that every $L$-local $(3, 2D_0)$-quasi-geodesic in $X$ is a $(\lambda, C)$-quasi-geodesic.
By Lemma~\ref{MorseLemma}, there exists a constant $B>0$ such that $d_{\mathrm{Haus}}(p, [p_-, p_+])\le B$ for every $(\lambda, C)$-quasi-geodesic $p$. 

\subsection{Existence of escaping hyperbolic elements in one direction}
By Theorem~\ref{geometricallyFiniteNearOrbit}, there exists a sequence $\{g_n\}$ of hyperbolic elements with axes $\gamma_n$ satisfying 
$$
\lim\limits_{n\to\infty} d(\gamma_n, Go)=\infty.
$$

After replacing each element by an appropriate conjugate, we may assume that $d(\gamma_n, o)$ is minimal among all its conjugates, i.e., $d(\gamma_n, o)=d(\gamma_n, Go)$.
By passing to a subsequence, we may further assume that $\lim\limits_{n\to\infty}g_n^+=\xi$ for some $\xi\in \partial X$, and that for every $n\in \mathbb{N}$,
$$d(\gamma_{n+1}, Go)> \sum\limits_{i=1}^n d(\gamma_i, Go)+\sum\limits_{i=1}^{n}\tau(g_i)+D_0 + D_1 +L + \lambda(2D_0+200\delta + B) +C .$$
See Figure \ref{fig:axes_sequence} for the illurstration.
\begin{figure}[htbp]
\centering
\begin{tikzpicture}[
    scale=1.5,
    point/.style={circle, fill=black, inner sep=1.2pt},
    boundarypoint/.style={circle, fill=red!70, inner sep=1.5pt},
    axis/.style={thick, blue!70},
    seqaxis/.style={thick, gray!50},
    segment/.style={thick, red!80, dashed},
    font=\small
]

% 单位圆（Poincaré圆盘边界）
\draw[thick] (0,0) circle (3);
\node at (0, -3.3) {$\partial X$};

% 边界点ξ（右侧）
\coordinate (xi) at (3,0);
\node[boundarypoint, label=right:$\xi$] at (xi) {};

% 参数：控制点偏移量
\def\d{1.5}

% 画一列趋于ξ的测地线（灰色）
% 角度从大到小，表示逐渐趋近
\foreach \angle in {50,40,30,20,10} {
    \pgfmathsetmacro{\cosangle}{3*cos(\angle)}
    \pgfmathsetmacro{\sinangle}{3*sin(\angle)}
    \coordinate (start) at ({\cosangle}, {\sinangle});
    \coordinate (end) at ({\cosangle}, {-\sinangle});
    \coordinate (ctrl) at ({\cosangle-\d}, 0);
    \draw[seqaxis] (start) .. controls (ctrl) .. (end);
}

% 突出两条特定的测地线：Axis(g_n) 和 Axis(g_{n+1})
% Axis(g_n) 对应角度40°
\pgfmathsetmacro{\angleN}{40}
\pgfmathsetmacro{\cosN}{3*cos(\angleN)}
\pgfmathsetmacro{\sinN}{3*sin(\angleN)}
\coordinate (startN) at ({\cosN}, {\sinN});
\coordinate (endN) at ({\cosN}, {-\sinN});
\coordinate (ctrlN) at ({\cosN-\d}, 0);
\draw[axis] (startN) .. controls (ctrlN) .. (endN);

% Axis(g_{n+1}) 对应角度30°
\pgfmathsetmacro{\angleNp}{30}
\pgfmathsetmacro{\cosNp}{3*cos(\angleNp)}
\pgfmathsetmacro{\sinNp}{3*sin(\angleNp)}
\coordinate (startNp) at ({\cosNp}, {\sinNp});
\coordinate (endNp) at ({\cosNp}, {-\sinNp});
\coordinate (ctrlNp) at ({\cosNp-\d}, 0);
\draw[axis] (startNp) .. controls (ctrlNp) .. (endNp);

% 在 Axis(g_n) 上取点 x_n，要求与ξ水平（y=0）
\path (startN) .. controls (ctrlN) .. (endN) coordinate[pos=0.5] (xn);
\node[point, label=left:$x_n$] at (xn) {};

% 在 Axis(g_{n+1}) 上取点 y_{n+1}，要求与ξ水平（y=0）
\path (startNp) .. controls (ctrlNp) .. (endNp) coordinate[pos=0.5] (yn);
\node[point, label=above right:$y_{n+1}$] at (yn) {};

% 连接 x_n 和 y_{n+1} 的线段（没有标签）
\draw[segment] (xn) -- (yn);

% 将 Axis(g_n) 标记往右下挪一点
\node[blue!70, above right] at (2.3, 1.9) {$\gamma_n$};

% 将 Axis(g_{n+1}) 标记往右上挪一点，但保持在 Axis(g_n) 的右下方
\node[blue!70, above right] at (2.6, 1.5) {$\gamma_{n+1}$};

\end{tikzpicture}
\caption{Choosing the escasping sequence of axes.}
\label{fig:axes_sequence}
\end{figure}

\begin{lemma}\label{contractingSystem}
    For any $g, h\in G$ and $n\ne m\in \mathbb{N}$, we have $d(g\gamma_n, h\gamma_m)>L+\lambda(2D_0+200\delta + B) + C$.
\end{lemma}
\begin{proof}
    Without loss of generality, assume that $n>m$.
    Observe that for any $m\in \mathbb{N}$ and $k\in \mathbb{Z}$, we have $d(g_m^k o, \gamma_m) \le d(\gamma_m, Go) + D_0$. 
    Consequently, for any point $x\in \gamma_m$, $d(x, Go)\le d(\gamma_m, Go)+\tau(g_m) + D_0+D_1$.
    
    The same estimate holds for any translate $h\gamma_m$. 
    Hence, applying the triangle inequality, for distinct indices $n> m\in \mathbb{N}$ and any $g, h\in G$, we obtain
    \begin{align}
    &d(g\gamma_n, h\gamma_m)\notag\\
    \ge &d(\gamma_n, Go)-\inf_{x\in h\gamma_m}d(x, Go)\notag\\
    \ge &d(\gamma_n, Go) - d(\gamma_m, Go) - \tau(g_m) - D_0 - D_1 \notag\\
    > &L+\lambda(2D_0+200\delta + B) + C.\qedhere \notag
    \end{align}
\end{proof}

Next, we fix reference points on these axes. For each $n\in \mathbb{N}$, let $x_n, y_{n+1}$ be points on $\gamma_n$ and $\gamma_{n+1}$, respectively, such that $d(x_n, y_{n+1})=d(\gamma_{n}, \gamma_{n+1})$.
For $n=0$, choose an arbitrary point $y_0\in \gamma_{0}$. 

The following result corresponds to \cite[Lemma 4.1]{KL19}. 
\begin{lemma}\label{qnFarAway}
    $$
    \lim\limits_{n\to\infty} d([x_n, y_{n+1}], Go)=\infty.
    $$
\end{lemma}
\begin{proof}
    For each $n\in \mathbb{N}$, let $p_n\in \pi_{\gamma_n}(o)$. Applying Lemma~\ref{fourPoints} to any $g\in G$, we obtain
    \begin{align}
        d(go, [x_n, y_{n+1}]) &\ge \min \left\{d(go, [x_n, p_n]), d(go, [p_n, y_{n+1}])\right\}-5\delta \notag \\
        &\ge \min \left\{d(go, [x_n, p_n]), d(go, [p_n, p_{n+1}]), d(go, [p_{n+1}, y_{n+1}])\right\} - 10\delta. \notag
    \end{align}
    Hence,
    $$
    d(Go, [x_n, y_{n+1}])\ge \min \left\{d(Go, [x_n, p_n]), d(Go, [p_n, p_{n+1}]), d(Go, [p_{n+1}, y_{n+1}])\right\} - 10\delta.
    $$
    Since $\lim\limits_{n\to\infty} d(\gamma_n, Go)=\infty$, and since $[x_n, p_n]\subset \gamma_n$ and $[p_{n+1}, y_{n+1}]\subset \gamma_{n+1}$, the terms $d(Go, [x_n, p_n])$ and $d(Go, [p_{n+1}, y_{n+1}])$ tend to infinity. See Figure \ref{fig:two_axes}.

% 需要在导言区添加以下包：
% \usepackage{tikz}
% \usetikzlibrary{calc}

\begin{figure}[htbp]
\centering
\begin{tikzpicture}[
    scale=1.2,
    point/.style={circle, fill=black, inner sep=1.2pt},
    axis/.style={thick, blue!70},
    segment/.style={thick, red!80, dashed},
    conn/.style={thin, black, densely dashed},
    font=\small
]

% 定义两条竖线的起点和终点（左边线比右边线高）
% 左边线 Axis(g_n) 的起点和终点
\coordinate (LnStart) at (-2, 2);   % 起点，稍高
\coordinate (LnEnd) at (-2, -1.5);  % 终点
% 右边线 Axis(g_{n+1}) 的起点和终点
\coordinate (RnStart) at (1, 1.5);  % 起点，稍低
\coordinate (RnEnd) at (1, -2);     % 终点

% 绘制两条略微向内弯的竖线（使用二次贝塞尔曲线，控制点使线向内弯曲）
% 左边线：控制点向右偏移，使其向右弯曲
\draw[axis] (LnStart) .. controls (-1.5, 0.25) .. (LnEnd);
\node[blue!70, above right] at ($(LnStart)+(0.2,0)$) {$\gamma_n$};

% 右边线：控制点向左偏移，使其向左弯曲
\draw[axis] (RnStart) .. controls (0.5, -0.25) .. (RnEnd);
\node[blue!70, above right] at ($(RnStart)+(0.2,0)$) {$\gamma_{n+1}$};

% 在左边线上标记点 p_n（中点）
\path (LnStart) .. controls (-1.5, 0.25) .. (LnEnd) coordinate[pos=0.5] (pn);
\node[point, label=left:$p_n$] at (pn) {};

% 在右边线上标记点 p_{n+1}（中点）
\path (RnStart) .. controls (0.5, -0.25) .. (RnEnd) coordinate[pos=0.5] (pnp1);
\node[point, label=right:$p_{n+1}$] at (pnp1) {};

% 在左边线上标记点 x_n（在 p_n 以下，例如 pos=0.7）
\path (LnStart) .. controls (-1.5, 0.25) .. (LnEnd) coordinate[pos=0.7] (xn);
\node[point, label=left:$x_n$] at (xn) {};

% 在右边线上标记点 y_{n+1}（在 p_{n+1} 以上，例如 pos=0.3）
\path (RnStart) .. controls (0.5, -0.25) .. (RnEnd) coordinate[pos=0.3] (yn);
\node[point, label=right:$y_{n+1}$] at (yn) {};

% 连接 x_n 和 y_{n+1} 的线段
\draw[segment] (xn) -- (yn);

% 添加点 o，放在左边线左侧，大约在竖直中间偏下位置
\coordinate (o) at (-3.5, -0.5);
\node[point, label=left:$o$] at (o) {};

% 连接 o 到 p_n 和 p_{n+1}
\draw[conn] (o) -- (pn);
\draw[conn] (o) -- (pnp1);

\end{tikzpicture}
\caption{}
\label{fig:two_axes}
\end{figure}

    It remains to analyze $d(Go, [p_n, p_{n+1}])$.
    First, note that 
    $$
    \lim\limits_{n\to\infty} g_n^+=\xi \quad \text{and} \quad \lim\limits_{n\to\infty} d(o, [p_n, g_n^+])=\infty,
    $$
    so the visual distance between $p_n$ and $g_n^+$ tends to $0$; that is,
    $$
    \lim\limits_{n\to\infty} p_n=\lim\limits_{n\to\infty} g_n^+=\xi.
    $$
    Consequently,
    $$
    \lim\limits_{n\to\infty} d(o, [p_n, p_{n+1}])=\infty.
    $$
    
    By construction, we have $d(p_n, o)=d(p_n, Go)$ and $d(p_{n+1}, o)=d(p_{n+1}, Go)$. We claim that for any $x\in [p_n, p_{n+1}]$,
    $$d(x, Go)\ge d(x, o)-2\delta.$$

    To prove this claim, fix $x\in [p_n, p_{n+1}]$ and choose $g\in G$ such that $d(x, go)=d(x, Go)$. Since the triangle $\triangle(xp_np_{n+1})$ is $\delta$-thin, there exists $y\in [o, p_n]\cup [o, p_{n+1}]$ with $d(x, y)\le \delta$. Without loss of generality, assume $y\in [o, p_n]$. Then,
    \begin{align}
        d(p_n, y)+d(y, o) &= d(p_n, o) \notag \\
        &\le d(p_n, go) \notag \\
        &\le d(p_n, x)+d(x, go) \notag \\
        &\le d(p_n, y)+d(x, y) + d(x, go) \notag \\
        &\le d(p_n, y)+\delta + d(x, go). \notag
    \end{align}
    Hence,
    \begin{align}
        d(x, go) &\ge d(y, o)-\delta \ge d(x, o)-d(x, y)-\delta \ge d(x, o)-2\delta, \notag
    \end{align}
    which proves the claim.

    The claim implies that 
    $$
    \lim\limits_{n\to\infty}d(Go, [p_n, p_{n+1}])=\lim\limits_{n\to\infty} d(o, [p_n, p_{n+1}])=\infty,
    $$ 
    completing the proof.
\end{proof}

\subsection{Construction of escaping geodesic rays}
By construction, we can choose $N_n\in \mathbb{N}$ such that $d(y_n, g_n^kx_n)\ge L$ for all $k\ge N_n$. Define
$$\mathcal{P}:=\{(k_n)\in \mathbb{N}^\mathbb{N}: k_n\ge N_n \text{ for every } n\in \mathbb{N}\}.$$

We define a map $\Phi: \mathcal{P}\to \partial X$ as follows. Given a sequence $(k_n)\in \mathcal{P}$, set
\begin{align}
p_n &= g_0^{k_0}\cdots g_{n-1}^{k_{n-1}}[y_n, g_n^{k_n}x_n], \notag \\
q_n &= g_0^{k_0}\cdots g_{n}^{k_n}[x_n, y_{n+1}], \notag \\
\gamma &=p_0\cup q_0\cup p_1\cup \cdots\cup p_n\cup q_n\cup \cdots. \notag
\end{align}

\begin{lemma}
    The infinite path $\gamma$ is an $L$-local $(3, 2D_0)$-quasi-geodesic, and therefore a $(\lambda, C)$-quasi-geodesic.
\end{lemma}
\begin{proof}
    Note that the Hausdorff distance between $p_n$ and a subpath of $\gamma_n$ is bounded by $D_0$, and the lengths of $p_n$ and $q_n$ are at least $L$.
    It suffices to prove the inequality $d(u, g_0^{k_0}\cdots g_{n}^{k_n}x_n)+d(g_0^{k_0}\cdots g_{n}^{k_n}x_n, v)\le 3d(u, v)+2D_0$ for any $u\in p_n$ and $v\in q_n$.

    By translation, it suffices to show that $d(u^\prime, x_n)+d(x_n, v^\prime)\le 3d(u^\prime, v^\prime)+2D_0$ for $u^\prime\in [g_n^{-k_n}y_n, x_n]$ and $v^\prime\in [x_n, y_{n+1}]$.
    Now
    \begin{align}
        &d(u^\prime, x_n)+d(x_n, v^\prime)\notag\\
        \le & d(u^\prime, v^\prime)+2d(v^\prime, x_n)\notag\\
        \le & d(u^\prime, v^\prime)+2\bigl(d(v^\prime, u^\prime)+d(u^\prime, \gamma_n)\bigr)\notag\\
        \le & 3d(u^\prime, v^\prime)+2D_0,\notag
    \end{align}
    which completes the proof.
\end{proof}

Intuitively, the path $\gamma$ is constructed by concatenating alternating segments: the segments $p_n$ travel along (translates of) the axis $\gamma_n$ for a prescribed number of steps $k_n$, while the segments $q_n$ are geodesic bridges connecting these axes. The condition in Lemma~\ref{contractingSystem} ensures these bridges are long enough, and the parameters are tuned so that the resulting path is a uniform quasi-geodesic. The endpoint $\gamma_+$ is forced to be non-conical because both types of segments, $p_n$ and $q_n$, eventually travel arbitrarily far from the orbit $Go$.

We define $\Phi((k_n))$ as the point in $\partial X$ represented by $\gamma$.

\begin{lemma}\label{imageNonconical}
    The image of $\Phi$ is contained in $\Lambda^{nc}G$.
\end{lemma}
\begin{proof}
    We first prove that the image lies in $\Lambda G$. Fix a sequence $(k_n)\in \mathcal{P}$. For each $n\in \mathbb{N}$, define
    $$
    \gamma_n=p_0\cup q_0\cup \cdots \cup q_{n-1}\cup (g_0^{k_0}\cdots g_{n-1}^{k_{n-1}})[y_n, g_n^+].
    $$
    Then $\gamma_n$ is also an $L$-local $(3, 2D_0)$-quasi-geodesic and hence a $(\lambda, C)$-quasi-geodesic. 
    Moreover, the positive endpoint $(\gamma_n)_+ = g_0^{k_0}\cdots g_{n-1}^{k_{n-1}}g_n^+$ belongs to $\Lambda G$.

    Since $(\gamma_n)_- = \gamma_-$ for $n\ge 1$, the visual distance between $(\gamma_n)_+$ and $\gamma_+$, as measured from $\gamma_-$, tends to $0$ as $n\to \infty$. 
    Thus, $\gamma_+\in \Lambda G$.

    It remains to show that $\gamma_+$ is not conical. 
    Note that the Hausdorff distance between $p_n$ and a subpath of a translate of $\gamma_n$ is at most $D_0$. 
    Since $d(\gamma_n, Go)\to\infty$, it follows that $d(p_n, Go)\to \infty$. 
    Because $q_n$ is a translate of $[x_n, y_{n+1}]$, Lemma~\ref{qnFarAway} implies $d(q_n, Go)\to\infty$. 
    Therefore, by the definition of conical points, $\gamma_+$ is non-conical.
\end{proof}

\begin{lemma}\label{projectiveNear}
For any $n\in \mathbb{N}$ we have
    $$
    \operatorname{diam}\bigl(\pi_{g_0^{k_0}\cdots g_{n-1}^{k_{n-1}}\gamma_n}( g_0^{k_0}\cdots g_{n}^{k_n}x_n)\cup \pi_{g_0^{k_0}\cdots g_{n-1}^{k_{n-1}}\gamma_n}(\gamma_+) \bigr)< 2D_0 + 200\delta.
    $$
\end{lemma}
\begin{proof}
    Let $x:=g_0^{k_0}\cdots g_{n}^{k_n}x_n$ be the point on $\gamma$, and let $p: [0, \infty)\to X$ be the subray of $\gamma$ starting at $x$. Set $\eta = g_0^{k_0}\cdots g_{n-1}^{k_{n-1}}\gamma_n$.
    Since $p$ begins with the segment $g_0^{k_0}\cdots g_{n}^{k_n}q_n$, and $\operatorname{Len}(q_n)\ge\lambda(2D_0+200\delta+B)+C$ by Lemma~\ref{contractingSystem}, the point $y:=p(\lambda(2D_0+200\delta+B)+C)$ lies in $g_0^{k_0}\cdots g_{n}^{k_n}q_n$.
    By Lemma~\ref{finiteHausdorffDistanceProjection} we have $\operatorname{diam}\bigl(\pi_{\eta}(y)\cup \pi_{\eta}(x)\bigr)\le 2D_0+100\delta$.

    Assume for contradiction that $\operatorname{diam}\bigl(\pi_{\eta}(x)\cup \pi_{\eta}(\gamma_+) \bigr)\ge 2D_0 + 200\delta$.
    Then $\operatorname{diam}\bigl(\pi_{\eta}(y)\cup \pi_{\eta}(\gamma_+)\bigr)\ge 100\delta\ge 24\delta$.
    By Corollary~\ref{projectionFellowTravel}, there exists $z\in [y, \gamma_+]$ such that $d(z, \pi_{\eta}(y))\le 24\delta$.
    Hence $d(z, x)\le d(z, \pi_{\eta}(y))+\operatorname{diam}\bigl(\pi_{\eta}(y)\cup \pi_{\eta}(x)\bigr)\le 2D_0+200\delta$.
    By the definition of the constant $B$, we have 
    $$d_{\mathrm{Haus}}(p([\lambda(2D_0+200\delta+B)+C,\infty)), [y, \gamma_+])\le B.$$
    Hence there exists $t\ge\lambda(2D_0+200\delta+B)+C$ such that $d(p(t), z)< B$.
    Therefore, $d(p(t), x)< 2D_0+200\delta + B$ for some $t\ge \lambda(2D_0+200\delta+B)+C$.
    However, since $\gamma$ (and thus $p$) is a $(\lambda, C)$-quasi-geodesic, we have 
    $$
    t\le \lambda d(p(t), p(0))+C<\lambda(2D_0+200\delta+B)+C,
    $$
    a contradiction.
\end{proof}

We choose a maximal subset $\mathcal{Q}\subset \mathcal{P}$ with the property that for any two distinct sequences $(k_n)$ and $(k_n^\prime)$ in $\mathcal{Q}$ and every $n\in \mathbb{N}$, either $k_n=k_n^\prime$ or $\lvert k_n-k_n^\prime \rvert\tau(g_n)> 4D_0 + 600\delta$. Clearly, $\mathcal{Q}$ has the cardinality of the continuum.

\begin{lemma}\label{injective}
    The restriction $\Phi|_{\mathcal{Q}}$ is injective.
\end{lemma}
\begin{proof}
    Take two distinct sequences $(k_n)$ and $(k_n^\prime)$ in $\mathcal{Q}$, and let $N$ be the minimal index for which $k_N \neq k_N^\prime$. Let $\eta = g_0^{k_0}\cdots g_{N-1}^{k_{N-1}}\gamma_{N}$. By the defining property of $\mathcal{Q}$, we have
    \begin{align}
        d(g_0^{k_0}\cdots g_{N}^{k_{N}}x_N, g_0^{k_0}\cdots g_{N}^{k^\prime_{N}}x_N) > 4D_0 + 600\delta.\notag
    \end{align}
    By Corollary~\ref{projectionFellowTravel}, we obtain
    \begin{align}
        d(\pi_{\eta}(g_0^{k_0}\cdots g_{N}^{k_{N}}x_N), \pi_{\eta}(g_0^{k_0}\cdots g_{N}^{k^\prime_{N}}x_N))>4D_0+500\delta.\notag
    \end{align}
    By Lemma~\ref{projectiveNear}, we have
    $$
    \operatorname{diam} \bigl(\pi_{\eta}(\gamma_+)\cup \pi_{\eta}(\gamma^\prime_+) \bigr)>100\delta.
    $$
    Therefore, $\gamma_+ \ne \gamma^\prime_+$, i.e., $\Phi((k_n))\ne \Phi((k_n^\prime))$.
\end{proof}

\subsection{Proof of Theorem~\ref{geometricallyFiniteNonconicalCountable}}
With injectivity established for the map $\Phi|_{\mathcal{Q}}$, we complete the proof of the main theorem.

%\begin{proof}[Proof of Theorem \ref{geometricallyFiniteNonconicalCountable}]
    If the action of $G$ on $X$ is geometrically finite, then every non-conical point in $\Lambda G$ is parabolic, and therefore the set of such points is countable.

    If the action is geometrically infinite, then by Lemma~\ref{injective}, the set $\Lambda^{nc}G$ contains the image $\Phi(\mathcal{Q})$, which has the cardinality of the continuum. Hence, $\Lambda^{nc}G$ is uncountable.
%\end{proof}

\iffalse
\subsection{Proof of Corollary~\ref{geometricallyFiniteConvergenceGroup}}
Suppose the action of $G$ on a compact metrizable space $M$ is a convergence group action. Let 
$$
\Theta^3(M)=\{(x_1, x_2, x_3)\in M^3: x_1\ne x_2, x_2\ne x_3, x_3\ne x_1\}
$$
be the space of distinct triples of points in $M$.
Then the diagonal action of $G$ on $\Theta^3(M)$ is proper.
Moreover, there exists a compatible $\delta$-hyperbolic metric on $\Theta^3(M)$ for some $\delta\ge0$, and the induced action on the boundary coincides with the original convergence group action; see \cite[Corollary 7.7]{sun2019dynamicalcharacterizationacylindricallyhyperbolic}. Therefore, by Theorem~\ref{geometricallyFiniteNonconicalCountable}, the proof is complete.
\fi

\bibliographystyle{alpha}
\bibliography{bibfile}

\end{document}